\documentclass[reqno]{amsart}

\usepackage{amsmath}
\usepackage{amssymb}

\newtheorem{Theorem}{Theorem}[section]
\newtheorem{Lemma}[Theorem]{Lemma}

\newtheorem{Corollary}[Theorem]{Corollary}

\theoremstyle{definition}
\newtheorem{Definition}[Theorem]{Definition}
\newtheorem{Remark}[Theorem]{Remark}

\numberwithin{equation}{section}

\newcommand\thref[1]{Theorem \ref{#1}}
\newcommand\leref[1]{Lemma \ref{#1}}

\newcommand\coref[1]{Corollary \ref{#1}}
\newcommand\deref[1]{Definition \ref{#1}}

\newcommand{\sld}{\mathfrak{sl}_{d+1}}
\newcommand{\gl}{\mathfrak{gl}}

\newcommand{\Nset}{\mathbb N}
\newcommand{\Cset}{\mathbb C}
\newcommand{\Rset}{\mathbb R}
\newcommand{\Iset}{\mathbb I}

\newcommand{\Pt}{\tilde{P}}
\newcommand{\Vt}{\tilde{V}}
\newcommand{\pt}{\tilde{p}}
\newcommand{\xt}{\tilde{x}}
\newcommand{\et}{\tilde{e}}
\newcommand{\mt}{\tilde{m}}
\newcommand{\kt}{\tilde{k}}
\newcommand{\nt}{\tilde{n}}
\newcommand{\Ht}{\tilde{H}}
\newcommand{\phit}{\tilde{\phi}}
\newcommand{\thetat}{\tilde{\theta}}
\newcommand{\MdN}{\mathcal{M}_{d,N}}

\newcommand{\Span}{\mathrm{span}}
\newcommand{\diag}{\mathrm{diag}}
\newcommand{\Ad}{\mathrm{Ad}}
\newcommand{\Id}{\mathrm{Id}}

\newcommand{\fa}{\mathfrak{a}}
\newcommand{\fb}{\mathfrak{b}}
\newcommand{\fK}{\mathcal{K}}

\newcommand{\fp}{\mathfrak{p}}

\newcommand{\pd}{{\partial}}

\newcommand{\cP}{\mathcal{P}}

\begin{document}

\title[A Lie interpretation of multivariate polynomials]{A Lie theoretic interpretation of multivariate hypergeometric  polynomials}

\date{September 2, 2011}

\author[P.~Iliev]{Plamen~Iliev$^*$}
\address{School of Mathematics, Georgia Institute of Technology,
Atlanta, GA 30332--0160, USA}
\email{iliev@math.gatech.edu}
\thanks{$^*$ Supported in part by NSF grant DMS-0901092 and Max-Planck-Institut f\"ur Mathematik, Bonn}

\maketitle

\begin{abstract}
In 1971 Griffiths used a generating function to define polynomials in $d$ variables orthogonal with respect to the multinomial distribution.  The polynomials possess a duality between the discrete variables and the degree indices. In 2004 Mizukawa and Tanaka related these polynomials to character algebras and the Gelfand hypergeometric series. Using this approach they clarified the duality and obtained a new proof of the orthogonality. In the present paper, we interpret these polynomials within the context of the Lie algebra $\sld(\Cset)$. Our approach yields yet another proof of the orthogonality. It also shows that the polynomials satisfy $d$ independent recurrence relations each involving $d^2+d+1$ terms. This combined with the duality establishes their bispectrality. We illustrate our results with several explicit examples.
\end{abstract}

\section{Introduction}\label{se1}
We start by introducing some basic notations which will be used 
throughout the paper. The first important object is the set $\fK_{d}$ which parametrizes families of multivariate Krawtchouk polynomials. 
In order to motivate the definition of $\fK_{d}$,  let us first consider positive real numbers $\{p_j\}_{j=0}^{d}$
such that $p_0+p_1+\cdots+p_d=1$ and let
$P=\diag(p_0,p_1,\dots,p_d)$ be the corresponding diagonal matrix. Let $w_0,w_1,\dots,w_d$ denote mutually orthogonal vectors in $\Rset^{d+1}$ with respect to the inner product  $\langle \alpha,\beta\rangle=\alpha^t P \beta$
such that $w_0=(1,1,\dots,1)^t$ and the $0$-th coordinate of $w_j$ is
$1$ for $j=1,2,\dots,d$. If we denote by $U$ the matrix with columns  $w_0,w_1,\dots,w_d$ then $Q=U^tPU$ is a nonsingular diagonal 
matrix whose $(0,0)$ entry is equal to $1$.
If we set $\Pt=p_0 Q^{-1}$, then 
$\Pt=\diag(p_0,\pt_1,\pt_2,\dots,\pt_d)$ is also diagonal and 
$$\frac{1}{p_0} PU\Pt U^{t}=I_{d+1},$$
where  $I_{d+1}$ denotes the identity $(d+1)\times (d+1)$ matrix. The above construction is closely 
related to the one by Griffiths \cite{G}.
It shows how to build matrices $U$ and $\Pt$ starting from $P$. 
In the present paper we will continue to work with the matrices 
$P, \Pt, U$ but we would like to  put $P$ and $\Pt$ on 
the same footing. 
Moreover, we would like to  work over $\Cset$ instead of $\Rset$. This leads to the following formal definition which extracts the important properties of the matrices $P, \Pt$  and $U$ needed in the paper. 
\begin{Definition}\label{de1.1} 
Let  $\fK_{d}$ denote 
the set of 4-tuples
$(\nu,P,\Pt,U)$,
where $\nu$ is a nonzero complex number 
and $P,\Pt, U$ are $(d+1)\times (d+1)$ matrices with complex entries
satisfying the following conditions:
\begin{itemize}
\item[{(i)}] $P=\diag(p_0,p_1,\dots,p_{d})$ and 
$\Pt=\diag(\pt_0,\pt_1,\dots,\pt_{d})$ are diagonal and
$p_0=\pt_0=\frac{1}{\nu}$;
\item[{(ii)}] $U=(u_{i,j})_{0\leq i,j\leq d}$ is such that 
$u_{0,j}=u_{j,0}=1$ for all $j=0,1,\dots,d$, i.e. 
\begin{equation}\label{1.1}
U=\left(\begin{matrix}1 & 1 & 1 & \dots &1\\ 
1 & u_{1,1} &u_{1,2} &\dots & u_{1,d}\\
 \vdots & \\
 1 & u_{d,1} & u_{d,2}  &\dots &u_{d,d}
\end{matrix}\right);
\end{equation}
\item[{(iii)}]  The following matrix equation holds
\begin{equation}\label{1.2}
\nu PU\Pt U^{t}=I_{d+1}. 
\end{equation}
\end{itemize}
\end{Definition}
We note that the points in $\fK_{d}$ arise naturally from the so called character algebras \cite{BI}. In particular, Hecke algebras of Gelfand pairs, or more generally, the Bose-Mesner algebras of commutative association schemes are examples of character algebras. From the above definition it is easy to see that $p_j$ and $\pt_j$ are nonzero numbers, such that
$$\sum_{j=0}^{d}p_j=\sum_{j=0}^{d}\pt_j=1.$$

Griffiths \cite{G} used a generating function
to construct $d$-variable Krawtchouk polynomials
for every point $\kappa\in\fK_{d}$ and positive integer $N$. His construction is as follows. 
For $m=(m_1,m_2,\dots,m_d)\in\Nset_0^d$ and $\mt=(\mt_1,\mt_2,\dots,\mt_d)\in\Nset_0^d$  such that $m_1+m_2+\cdots+m_d\leq N$,  $\mt_1+\mt_2+\cdots+\mt_d\leq N$ define polynomials $\cP(m,\mt;\kappa,N)=\cP(m,\mt)$ in the variables $\mt_1,\dots,\mt_d$  with degree indices $m_1,\dots,m_d$ by
$$\prod_{i=0}^{d}\left(1+\sum_{j=1}^{d}u_{i,j}z_j\right)^{\mt_i}=\sum_{\begin{subarray}{c} m \in\Nset_0^{d}\\ m_1+\cdots +m_d \leq N\end{subarray}}
\frac{N!}{m_0!m_1!\cdots m_d!}\cP(m,\mt)z_1^{m_1}\cdots z_d^{m_d},$$
where $m_0=N-m_1-m_2-\cdots-m_d$ and $\mt_0=N-\mt_1-\mt_2-\cdots-\mt_d$.

Mizukawa and Tanaka \cite{MT} gave an explicit formula for $\cP(m,\mt)$
in terms of the Gelfand hypergeometric series 
\begin{equation}\label{1.3}
\cP(m,\mt)=\sum_{A=(a_{i,j})\in\MdN}\frac{\prod_{j=1}^{d}(-m_j)_{\sum_{i=1}^{d}a_{i,j}}\, \prod_{i=1}^{d}(-\mt_i)_{\sum_{j=1}^{d}a_{i,j}}}{(-N)_{\sum_{i,j=1}^{d}a_{i,j}}}
\prod_{i,j=1}^{d}\frac{\omega_{i,j}^{a_{i,j}}}{a_{i,j}!},
\end{equation}
where $\omega_{i,j}=1-u_{i,j}$. In the above formula
$\MdN$ denotes the set of all
$d\times d$ matrices $A=(a_{i,j})$ with
nonnegative integer entries, 
such that $\sum_{i,j=1}^{d}a_{i,j}\leq N$. 
One advantage of equation \eqref{1.3} is the transparent symmetry between $m$ and $\mt$. 

In the present paper, we define two Cartan subalgebras for $\sld(\Cset)$, denoted $H$ and $\Ht$, which depend on $\kappa$. We display an antiautomorphism $\fa$ of $\sld(\Cset)$ that fixes each element of $H$ and each element of $\Ht$. We consider a certain finite-dimensional irreducible $\sld(\Cset)$-module $V$ consisting of homogeneous polynomials in $(d+1)$ variables of total degree $N$. We define a nondegenerate symmetric bilinear form $\langle \,,\,\rangle$ on $V$ such that $\langle \beta. \xi,\eta\rangle = \langle \xi,\fa(\beta). \eta\rangle$
for all $ \beta \in\sld(\Cset)$ and $\xi,\eta \in V$. We define also two bases for $V$; one diagonalizes $H$ and the other diagonalizes $\Ht$. Both bases are orthogonal with respect to $\langle \,,\,\rangle$.
We show that when $\langle\,,\,\rangle$ is applied to a vector in each basis, the result is 
a trivial factor times one of the polynomials $\cP(m,\mt)$ defined in \eqref{1.3}. Thus the transition matrices between the bases are described by the polynomials \eqref{1.3}. From these results we recover the previously known orthogonality relation proved in \cite{G,MT}. 

Our approach naturally leads to two sets of $d$ recurrence relations for the polynomials \eqref{1.3}, parametrized by the Cartan algebras $H$ and $\Ht$. Equivalently, we can interpret these recurrence relations as two commutative algebras of difference operators, each generated by $d$ algebraically independent operators diagonalized by the polynomials $\cP(m,\mt)$. One of the algebras consists of difference operators acting on the variables $m$ and the other one on the variables $\mt$.  Following the terminology established in the literature starting with \cite{DG} we can say that the polynomials $\cP(m,\mt)$ solve a discrete-discrete bispectral problem. 
The two algebras can be connected via a natural (bispectral) involution $\fb$  on $\fK_{d}$ defined as follows:
$$\fb:\kappa=(\nu,P,\Pt,U)\rightarrow \fb(\kappa)=(\nu,\Pt,P,U^{t}).$$
This map exchanges the roles $m$ and $\mt$ and corresponds to a symmetry between the degree parameters and the variables of the polynomials.

At the end of the paper, we illustrate how our theory applies to multivariate Krawtchouk polynomials, which have appeared in different applications in the literature. The first example explains how the present work extends our joint paper with Terwilliger \cite{IT} related to the bivariate polynomials defined in \cite{HR}. The next example concerns the polynomials defined by Milch \cite{M} more than 40 years ago within the context of multi-dimensional growth birth and death processes. We fix the matrix $P$ and we exhibit explicit matrices $\Pt$ and $U$, satisfying the conditions in \deref{de1.1}. Our constructions lead to the bispectral commutative algebras of difference operators obtained recently in \cite[Section~5.4]{GI}. The bispectral involution $\fb$ above corresponds to the mapping $\mathfrak{f}$ in \cite[page~450, eq. (5.22)]{GI}. In the last example we explain how the polynomials discussed in \cite{DS} fit within our framework.

The results of the present paper  yield a family of solutions to Problem~7.1 in \cite{IT}. Each of these solutions can be viewed as a rank $d$ generalization of Leonard pair. The Leonard pairs are defined and classified in \cite{T1}. For more information on Leonard pairs see \cite{T2}.

\section{Cartan subalgebras of $\sld(\Cset)$}\label{se2}

For $i,j\in\{0,1,\dots,d\}$, let $e_{i,j}$ denote the $(d+1)\times(d+1)$ matrix that has $(i,j)$-entry $1$ and all other entries $0$. We denote by $H$ the standard Cartan subalgebra of $\sld(\Cset)$ consisting of all diagonal matrices with basis $\{\phi_1,\phi_2,\dots,\phi_{d}\}$ where 
\begin{equation}\label{2.1}
\phi_{i}=e_{i,i}-\frac{1}{d+1}I_{d+1}, \text{ for }i=1,2,\dots,d.
\end{equation}

Let $R$ denote the $(d+1)\times(d+1)$ matrix given by 
\begin{equation}\label{2.2}
R=\thetat \Pt U^{t},
\end{equation}
where $\thetat\in\Cset$ is such that $\det(R)=1$. We consider the automorphism $\Ad_R$ on $\sld(\Cset)$ defined by $\Ad_R(\beta)=R\beta R^{-1}$ for every $\beta\in\sld(\Cset)$.  We denote by $\phit_k$ and $\et_{i,j}$ the images of $\phi_k$ and $e_{i,j}$ under $\Ad_R$, i.e. we set
\begin{equation}\label{2.3}
\phit_{k}=R\phi_kR^{-1}, \quad \et_{i,j}=Re_{i,j}R^{-1}
\end{equation}
for $k=1,2,\dots,d$ and $i,j\in\{0,1,\dots,d\}$. Let 
\begin{equation}\label{2.4}
\Ht=\Span\{\phit_1,\phit_2,\dots,\phit_d\}
\end{equation}
denote the conjugated Cartan subalgebra of $\sld(\Cset)$. From \eqref{1.2} and \eqref{2.2} we see that 
\begin{equation}\label{2.5}
R^{-1}=\theta PU, \text{ where }\theta=\frac{\nu}{\thetat}.
\end{equation}
Using the above equations we can easily expand $\phit_i$ in terms of the basis 
$\{\phi_j,e_{k,l}\}$, $j\in\{1,2,\dots,d\}$, $k\neq l\in\{0,1,\dots,d\}$ of $\sld(\Cset)$ as follows
\begin{equation}\label{2.6}
\phit_i=\nu\sum_{0\leq k\neq l\leq d}p_{i}\pt_{k}u_{i,k}u_{i,l}e_{k,l}+\sum_{j=1}^{d}p_{i}(\nu \pt_{j}u_{i,j}^2-1)\phi_j,
\end{equation}
for $i=1,2,\dots,d$. Likewise, we can also expand $\phi_i$ in terms of the dual basis $\{\phit_j,\et_{k,l}\}$, $j\in\{1,2,\dots,d\}$, $k\neq l\in\{0,1,\dots,d\}$ of $\sld(\Cset)$:
\begin{equation}\label{2.7}
\phi_i=\nu\sum_{0\leq k\neq l\leq d}\pt_{i}p_{k}u_{k,i}u_{l,i}\et_{k,l}+\sum_{j=1}^{d}\pt_{i}(\nu p_{j}u_{j,i}^2-1)\phit_j,
\end{equation}
for $i=1,2,\dots,d$.

Next we define an antiautomorphism $\fa$ on $\sld(\Cset)$ by 
\begin{equation}\label{2.8}
\fa(\beta)=\Pt\beta^{t} \Pt^{-1} \text{ for every }\beta\in\sld(\Cset).
\end{equation}
Note that $\fa$ is an involution, i.e. $\fa\circ \fa=\Id$. Using that fact that $P$ and $\Pt$ are diagonal matrices and \eqref{1.2} one can check the following lemma.
\begin{Lemma}\label{le2.1}
We have
\begin{align}
\fa(\phi_i)&=\phi_i & \fa(\phit_i)&=\phit_i& &\text{ for all }i=1,2,\dots,d\label{2.9}\\
\fa(e_{i,j})&=\frac{\pt_j}{\pt_i}e_{j,i} & \fa(\et_{i,j})&=\frac{p_j}{p_i}\et_{j,i} & &\text{ for all }0\leq i\neq j\leq d.\label{2.10}
\end{align}
In particular, $\fa$ preserves the Cartan subalgebras $H$ and $\Ht$. 
\end{Lemma}
\begin{Lemma}\label{le2.2}
The Cartan subalgebras $H$ and $\Ht$ together generate $\sld(\Cset)$.
\end{Lemma}
\begin{proof}
Consider
$$\phi_0=-\sum_{j=1}^{d}\phi_j=e_{0,0}-\frac{1}{d+1}I_{d+1}\in H$$
and
$$\phit_0=-\sum_{j=1}^{d}\phit_j=Re_{0,0}R^{-1}-\frac{1}{d+1}I_{d+1}\in \Ht.$$
Using equations \eqref{2.2}, \eqref{2.5} and the explicit form of the matrices $P$, $\Pt$ and $U$ we see that 
\begin{align*}
\phit_0&= \nu\Pt U^{t}e_{0,0}PU-\frac{1}{d+1}I_{d+1}\\
&=\left(\begin{matrix}\pt_{0} & \pt_{0} & \pt_{0} & \dots &\pt_{0}\\ 
\pt_{1} & \pt_{1} &\pt_{1} &\dots & \pt_{1}\\
 \vdots & \\
\pt_{d} & \pt_{d} & \pt_{d} & \dots &\pt_{d}
\end{matrix}\right)-\frac{1}{d+1}I_{d+1}.
\end{align*}
From the last equation it follows easily that for $0\leq i\neq j\leq d$ we have
$$e_{i,j}=\frac{[\phi_{j},[\phi_{i},[\phi_{j},\phit_0]]]-[\phi_{i},[\phi_{j},\phit_0]]}{2\pt_i},$$
completing the proof.
\end{proof}

\section{An $\sld(\Cset)$-module}\label{se3}
Let $x_0,x_1,\dots,x_{d}$ denote mutually commuting variables. We set 
$x=(x_0,x_1,\dots,x_{d})$ and we denote by $v_0,v_1,\dots,v_{d}$ the standard basis of 
$\Cset^{d+1}$. Throughout the paper we use multi-index notation. For instance, $\Cset[x]$ stands for the $\Cset$-algebra consisting of all polynomials in 
$x_0,x_1,\dots,x_{d}$ and if $\lambda=(\lambda_0,\lambda_1,\dots,\lambda_{d})\in\Nset_0^{d+1}$ then 
$$x^{\lambda}=x_0^{\lambda_0}x_1^{\lambda_1}\cdots x_{d}^{\lambda_{d}},\quad
\lambda!=\lambda_0!\lambda_1!\cdots\lambda_{d}!,$$
$|\lambda|=\lambda_0+\lambda_1+\cdots+\lambda_{d}$, etc. We shall denote by $p=(p_0,p_1,\dots,p_{d})$ and  $\pt=(\pt_0,\pt_1,\dots,\pt_{d})$ the vectors whose components are the entries of the diagonal matrices $P$ and $\Pt$ in \deref{de1.1}.

Consider the representation  $\rho:\sld(\Cset)\rightarrow \gl(\Cset[x])$ of $\sld(\Cset)$ defined by
\begin{equation}\label{3.1}
\begin{split}
&\rho(e_{i,j})=x_i\pd_{x_{j}}\\
&\rho(e_{i,i}-e_{j,j})=x_i\pd_{x_{i}}-x_j\pd_{x_{j}},
\end{split}
\end{equation}
for $i\neq j$.
Thus every element of $\sld(\Cset)$ acts as a derivation. Let 
$$\Iset=\{\lambda\in\Nset_0^{d+1}:|\lambda|=N\}$$
and let $V$ denote the subspace of $\Cset[x]$ consisting of homogeneous polynomials of total degree $N$, i.e.
\begin{equation}\label{3.2}
V=\Span\{x^{\lambda}:\lambda\in\Iset\}.
\end{equation}
Since the operators in \eqref{3.1} preserve the total degree of the polynomials, 
we can consider $V$ as an $\sld(\Cset)$-submodule of $\Cset[x]$ and it is easy to see that this module 
is irreducible. In particular, from \leref{le2.2} it follows that there is no proper subspace 
$W$ of $V$ such that $HW\subseteq W$ and  $\Ht W\subseteq W$.

Since
\begin{align}
\phi_i.x^{\lambda}&= \left(\lambda_{i}-\frac{N}{d+1}\right)x^{\lambda}& &\text{for }i=1,2,\dots,d,\label{3.3}\\
e_{i,j}.x^{\lambda}&= \lambda_{j}x^{\lambda+v_i-v_j} &&\text{for }0\leq i\neq j\leq d,\label{3.4}
\end{align}
we see that 
\begin{equation}\label{3.5}
V=\bigoplus_{\lambda\in\Iset}V_{\lambda}, \text{ where }V_{\lambda}=\Span\{x^{\lambda}\}
\end{equation}
is the weight decomposition of $V$ with respect to $H$. To describe the weight decomposition of $V$ with respect to $\Ht$ we need to make a change of variables. Let us define $\xt=(\xt_0,\xt_1,\dots,\xt_{d})$ by
\begin{equation}\label{3.6}
\xt=x R,
\end{equation}
where $R$ is the matrix given in \eqref{2.2}. From equations \eqref{2.3}, \eqref{3.3} and \eqref{3.4} it follows that 
\begin{align}
\phit_i.\xt^{\lambda}&= \left(\lambda_{i}-\frac{N}{d+1}\right)\xt^{\lambda}& &\text{for }i=1,2,\dots,d,\label{3.7}\\
\et_{i,j}.\xt^{\lambda}&= \lambda_{j}\xt^{\lambda+v_i-v_j} &&\text{for }0\leq i\neq j\leq d.\label{3.8}
\end{align}
Thus
\begin{equation}\label{3.9}
V=\bigoplus_{\lambda\in\Iset} \Vt_{\lambda}, \text{ where }\Vt_{\lambda}=\Span\{\xt^{\lambda}\}
\end{equation}
is the weight decomposition of $V$ with respect to $\Ht$. 

\begin{Remark}\label{re3.1}
We comment on how $H$ and $\Ht$ act 
on the  weight spaces of the other one.
A pair of elements $\lambda$ and $\mu$ in $\Iset$
will be called {\it adjacent} whenever
$\lambda-\mu$ is a permutation of
$(1,-1,0,0,\dots,0)\in\Cset^{d+1}$.
Then $H$ and $\Ht$ act on the weight spaces of the
other one as follows.
For all $\lambda\in\Iset$,
\begin{equation*}
\Ht V_\lambda \subseteq  V_\lambda + 
\sum_{\genfrac{}{}{0pt}{} {\mu \in \mathbb I}{\mu \;{\rm adj}\;\lambda}}
V_\mu,
\qquad \qquad 
H \Vt_\lambda \subseteq  \Vt_\lambda + 
\sum_{\genfrac{}{}{0pt}{} {\mu \in \mathbb I}{\mu \;{\rm adj}\;\lambda}}
\Vt_\mu.
\end{equation*}
\end{Remark}

We conclude the section by writing explicit formulas for the entries of $\xt$, which will be needed later. Using \eqref{3.6}, the notations in \deref{de1.1} and \eqref{2.2} we obtain
\begin{align}
&\xt_0=\thetat \sum_{j=0}^{d}\pt_jx_j,\label{3.10}\\
&\xt_{k}=\thetat\pt_0x_0+\thetat \sum_{j=1}^{d}u_{k,j}\pt_jx_j \text{ for }k=1,2,\dots,d.\label{3.11}
\end{align}
Recall that $\omega_{i,j}=1-u_{i,j}$ and therefore, subtracting \eqref{3.10} from \eqref{3.11} we 
can rewrite equations \eqref{3.11} as follows
\begin{equation}\label{3.12}
\xt_{k}=\xt_0-\thetat \sum_{j=1}^{d}\pt_{j}\omega_{k,j}x_{j}, \text{ for }k=1,2,\dots,d.
\end{equation}

\section{A bilinear form}\label{se4}
We define a symmetric bilinear form $\langle,\rangle$ on $V$ by
\begin{equation}\label{4.1}
\langle x^n, x^m\rangle=\delta_{n,m}\frac{n!}{\pt^{n}}\theta^{N} \text{ for all }n,m\in\Iset.
\end{equation}
Using the explicit formulas for the action of $\fa$ on the basis $\{\phi_j,e_{k,l}\}$, $j\in\{1,2,\dots,d\}$, $k\neq l\in\{0,1,\dots,d\}$ of $\sld(\Cset)$ in \leref{le2.1} it is easy to check that 
\begin{equation}\label{4.2}
\langle  \beta. \xi, \eta \rangle=\langle  \xi, \fa(\beta). \eta \rangle \text{ for all }\beta\in\sld(\Cset), \quad \xi,\eta\in V.
\end{equation}
We prove next that the vectors $\xt^{n}$ are also mutually orthogonal with respect to $\langle,\rangle$.
\begin{Lemma}\label{le4.1}
For $n,m\in\Iset$ we have
\begin{equation}\label{4.3}
\langle \xt^n, \xt^m\rangle=\delta_{n,m}\frac{n!}{p^{n}}\thetat^{N}.
\end{equation}
\end{Lemma}
\begin{proof}
For $n\neq m$, equation \eqref{4.3} follows immediately from \eqref{4.2}, the fact that $\fa$ preserves $\Ht$ and \eqref{3.7}. Thus it remains to  consider the case $n=m$, i.e. we need to show that
\begin{equation}\label{4.4}
||\xt^n||^2=\langle \xt^n, \xt^n\rangle=\frac{n!}{p^{n}}\thetat^{N}.
\end{equation}
 We prove that \eqref{4.4} is true by induction on 
$n_1+n_2+\cdots+n_{d}$.
If $n_1+n_2+\cdots+n_{d}=0$ then $n=(N,0,0,\dots,0)$ and we need to show that 
$||\xt_0^N||^2=N!\thetat^{N}/p_0^N$. 
Taking the $N$th power of equation \eqref{3.10} and using \eqref{4.1} we find:
\begin{align*}
||\xt_0^N||^2&=\thetat^{2N}\left\| \left(\sum_{j=0}^{d}\pt_{j}x_j \right)^N\right\|^2 
=\thetat^{2N}\left\| \sum_{\alpha\in\Iset}\frac{N!}{\alpha!}\pt^{\alpha}x^{\alpha}\right\| ^2\\
&=\thetat^{2N}\sum_{\alpha\in\Iset}\frac{(N!)^2}{(\alpha!)^2}\pt^{2\alpha}||x^{\alpha}||^2
=\thetat^{2N}\sum_{\alpha\in\Iset}\frac{(N!)^2}{\alpha!}\pt^{\alpha}\theta^N\\
&=\thetat^{N}(\thetat\theta)^NN!\sum_{\alpha\in\Iset}\frac{N!}{\alpha!}\pt^{\alpha}
=\frac{\thetat^N}{p_0^N}N!,
\end{align*}
completing the proof in this case. Suppose now $n_1+n_2+\cdots+n_{d}>0$. Then 
$n_j>0$ for some $j\in\{1,\dots,d\}$ and we consider the equality
\begin{equation*}
\langle \et_{0,j}.\xt^n,\xt^{n+v_0-v_j} \rangle = \langle \xt^n,\fa(\et_{0,j}).\xt^{n+v_0-v_j} \rangle.
\end{equation*}
From \eqref{3.8} it follows that the left-hand side equals $n_j||\xt^{n+v_0-v_j}||^2$. From \eqref{2.10} and \eqref{3.8} we see that the right-hand side is $\frac{p_j}{p_0}(n_0+1)||\xt^{n}||^2$. Therefore, we have
\begin{equation*}
||\xt^{n}||^2=\frac{p_0}{p_j}\frac{n_j}{(n_0+1)}||\xt^{n+v_0-v_j}||^2,
\end{equation*}
completing the proof by induction.
\end{proof}

\section{The Krawtchouk polynomials and $\sld(\Cset)$}\label{se5}

We are now ready to prove that the transition matrices between the bases $\{x^n:n\in\Iset\}$ and 
$\{\xt^n:n\in\Iset\}$ can be expressed in terms of the multivariate Krawtchouk polynomials \eqref{1.3}. 
To state the theorem we shall use the following convention: for a vector 
$w=(w_0,w_1,\dots,w_{d})\in\Cset^{d+1}$ we denote by $w'=(w_1,\dots,w_{d})\in\Cset^{d}$.

\begin{Theorem}\label{th5.1}
With the above convention we have
\begin{subequations}\label{5.1}
\begin{align}
&\frac{\xt^{\nt}}{\thetat^N}=N!\sum_{n\in\Iset}\cP(n',\nt')\frac{\pt^{n}}{n!}x^{n}\quad \text{ for }\nt\in\Iset\label{5.1a}\\
\intertext{and}
&\frac{x^{n}}{\theta^N}=N!\sum_{\nt\in\Iset}\cP(n',\nt')\frac{p^{\nt}}{\nt!}x^{\nt}\quad \text{ for }n\in\Iset. \label{5.1b}
\end{align}
\end{subequations}
\end{Theorem}

\begin{proof}
Let us denote $\omega_i=(\omega_{i,1},\omega_{i,2},\dots,\omega_{i,d})$. 
Using equation \eqref{3.12} we find
\begin{align*}
\frac{\xt_{i}^{\nt_{i}}}{\thetat^{\nt_{i}}}&= \sum_{\begin{subarray}{c}c_i\in\Nset_0^{d}\\ |c_i|\leq \nt_{i}\end{subarray}}\frac{\nt_{i}!}{c_i!(\nt_{i}-|c_i|)!}\left(\frac{\xt_0}{\thetat}\right)^{\nt_{i}-|c_i|}\,(-1)^{|c_i|}(\pt')^{c_i}\omega_i^{c_i}(x')^{c_i}\\
&= \sum_{c_i\in\Nset_0^{d}}\frac{(-\nt_{i})_{|c_i|}}{c_i!}\left(\frac{\xt_0}{\thetat}\right)^{\nt_{i}-|c_i|}(\pt')^{c_i}\omega_i^{c_i}(x')^{c_i},
\text{ for }i=1,2,\dots,d.
\end{align*}
If we set $c_i=(a_{i,1},a_{i,2},\dots,a_{i,d})$ and denote by $A$ the matrix with entries $(a_{i,j})$, then multiplying the above equations we obtain
\begin{equation}\label{5.2}
\begin{split}
\frac{\xt^{\nt}}{\thetat^N}=&\sum_{A\in\MdN} \left(\frac{\xt_0}{\thetat}\right)^{N-\sum_{i,j=1}^{d}a_{i,j}}
\prod_{i=1}^{d} (-\nt_{i})_{\sum_{j=1}^{d}a_{i,j}}  \prod_{j=1}^{d}(\pt_{j}x_{j})^{\sum_{i=1}^{d}a_{i,j}}\\
&\quad \times \prod_{i,j=1}^{d}\frac{\omega_{i,j}^{a_{i,j}}}{a_{i,j}!}.
\end{split}
\end{equation}
Using now \eqref{3.10} we get
\begin{equation}\label{5.3}
\left(\frac{\xt_0}{\thetat}\right)^{N-\sum_{i,j=1}^{d}a_{i,j}}=\sum_{\begin{subarray}{c}\gamma\in\Nset_0^{d+1}\\ |\gamma|= N-\sum_{i,j=1}^{d}a_{i,j}\end{subarray}}\frac{(N-\sum_{i,j=1}^{d}a_{i,j})!}{\gamma!}\prod_{k=0}^{d}(\pt_kx_k)^{\gamma_k}.
\end{equation}
Note that 
$$\frac{(N-\sum_{i,j=1}^{d}a_{i,j})!}{\gamma!}=\frac{N!}{\gamma_0! \prod_{j=1}^{d}(\gamma_{j}+\sum_{i=1}^{d}a_{i,j})!}\frac{\prod_{j=1}^{d}(-\gamma_{j}-\sum_{i=1}^{d}a_{i,j})_{\sum_{i=1}^{d}a_{i,j}}}{(-N)_{\sum_{i,j=1}^{d}a_{i,j}}}.$$
If we change the variables as follows $n_0=\gamma_0$ and $n_{j}=\gamma_{j}+\sum_{i=1}^{d}a_{i,j}$ for $j=1,2,\dots,d$ and we substitute the last formula in \eqref{5.3} we obtain
\begin{equation}\label{5.4}
\begin{split}
\left(\frac{\xt_0}{\thetat}\right)^{N-\sum_{i,j=1}^{d}a_{i,j}}=&\frac{1}{(-N)_{\sum_{i,j=1}^{d}a_{i,j}}}\sum_{n\in\Iset}\frac{N!}{n!}\prod_{j=1}^{d}(-n_j)_{\sum_{i=1}^{d}a_{i,j}}\\
&\quad\times(\pt_0x_0)^{n_0}\prod_{j=1}^{d}(\pt_{j}x_{j})^{n_{j}-\sum_{i=1}^{d}a_{i,j}}.
\end{split}
\end{equation}
The proof of \eqref{5.1a} now follows immediately by plugging \eqref{5.4} in \eqref{5.2}. The proof of 
\eqref{5.1b} can be obtained along the same lines, or one can use the duality.
\end{proof}
\begin{Corollary}\label{co5.2}
For $n,\nt \in \Iset$ we have
\begin{equation}\label{5.5}
\cP(n',\nt')=\frac{1}{\nu^{N}N!}\langle x^n,\xt^{\nt}\rangle.
\end{equation}
\end{Corollary}
As an immediate consequence of \thref{th5.1} and \coref{co5.2}, we obtain a new proof of the orthogonality relations established in \cite{G,MT}.
\begin{Corollary}\label{co5.3}
For $\nt,\kt\in \Iset$ we have
\begin{subequations}
\begin{equation}\label{5.6a}
N! \sum_{n\in\Iset}\cP(n',\nt')\cP(n',\kt')\frac{\pt^{n}}{n!}=\delta_{\nt,\kt}\frac{\nt!}{N!\,\nu^{N}\,p^{\nt}},
\end{equation}
and similarly for $n,k\in\Iset$ we have
\begin{equation}\label{5.6b}
N! \sum_{\nt\in\Iset}\cP(n',\nt')\cP(k',\nt')\frac{p^{\nt}}{\nt!}=\delta_{n,k}\frac{n!}{N!\,\nu^{N}\,\pt^{n}}.
\end{equation}
\end{subequations}
\end{Corollary}

\begin{proof}
Equation \eqref{5.6a} follows if we consider the inner product of each side of equation \eqref{5.1a} with 
$\xt^{\kt}$ and then we use \eqref{4.3} and \eqref{5.5}. The proof of \eqref{5.6b} is similar.
\end{proof}

\section{Bispectral difference equations}\label{se6}

In this section we show that the Cartan subalgebras $H$ and $\Ht$ lead naturally to recurrence relations for the multivariate Krawtchouk polynomials. We work below with $\Cset^d$ and we denote by 
$v_1,v_2,\dots,v_{d}$ its standard basis.

\begin{Theorem}\label{th6.1}
Fix $m=(m_1,m_2,\dots,m_d)\in\Nset_0^{d}$  and $\mt=(\mt_1,\mt_2,\dots,\mt_d)\in\Nset_0^{d}$ such that $|m|\leq N$ and $|\mt|\leq N$.
Then for every $i\in\{1,2,\dots,d\}$ we have
\begin{subequations}\label{6.1}
\begin{align}
&\left(m_{i}-\frac{N}{d+1}\right)\cP(m,\mt)=\sum_{l=1}^{d}\pt_{i}u_{l,i}\mt_{l}\cP(m,\mt-v_l)\nonumber\\
&\qquad
+\nu\sum_{k=1}^{d}\pt_{i}p_{k}u_{k,i}(N-|\mt|)\cP(m,\mt+v_k)\nonumber\\
&\qquad
+\sum_{j=1}^{d}\pt_{i}(\nu p_{j}u_{j,i}^2-1)\left(\mt_{j}-\frac{N}{d+1} \right)\cP(m,\mt)\nonumber\\
&\qquad
+\nu\sum_{1\leq k\neq l\leq d}\pt_{i}p_{k}u_{k,i}u_{l,i}\mt_{l}\cP(m,\mt+v_k-v_l),\label{6.1a}
\end{align}
and 
\begin{align}
&\left(\mt_{i}-\frac{N}{d+1}\right)\cP(m,\mt)=\sum_{l=1}^{d}p_{i}u_{i,l}m_{l}\cP(m-v_l,\mt)\nonumber\\
&\qquad
+\nu\sum_{k=1}^{d}p_{i}\pt_{k}u_{i,k}(N-|m|)\cP(m+v_k,\mt)\nonumber\\
&\qquad
+\sum_{j=1}^{d}p_{i}(\nu\pt_{j}u_{i,j}^2-1)\left(m_{j}-\frac{N}{d+1} \right)\cP(m,\mt)\nonumber\\
&\qquad
+\nu\sum_{1\leq k\neq l\leq d}p_{i}\pt_{k}u_{i,k}u_{i,l}m_{l}\cP(m+v_k-v_l,\mt).\label{6.1b}
\end{align}
\end{subequations}
\end{Theorem}
\begin{proof}
For $m,\mt\in\Nset_0^{d}$ such that $|m|\leq N$ and $|\mt|\leq N$, we consider 
$n=(N-|m|,m_1,m_2,\dots,m_d)\in\Iset$ and $\nt=(N-|\mt|,\mt_1,\mt_2,\dots,\mt_d)\in\Iset$. 
From \leref{le2.1} and \eqref{4.2} it follows that
\begin{equation}\label{6.2}
\langle  \phi_i. x^{n},\xt^{\nt} \rangle =\langle  x^{n},\phi_i. \xt^{\nt} \rangle.
\end{equation}
From \eqref{3.3} and \coref{co5.2} we see that the left-hand side of \eqref{6.2} is equal to 
$\nu^N N! $ times the left-hand side of \eqref{6.1a}. Now we use \eqref{2.7} and \leref{le2.1} to evaluate 
the right-hand side of equation \eqref{6.2}, which gives $\nu^N N! $ times the right-hand side of \eqref{6.1a}. The proof of 
\eqref{6.1b} is similar.
\end{proof}

\begin{Remark}\label{re6.2}
\thref{th6.1} establishes the bispectrality of the polynomials $\cP(m,\mt)$. Indeed, equations \eqref{6.1a} show that the polynomials $\cP(m,\mt)$ are common eigenfunctions of $d$ difference operators acting on the variables $\mt_1,\mt_2,\dots,\mt_d$, with coefficients independent of $m_1,m_2,\dots,m_d$, while equations \eqref{6.1b} indicate that the same polynomials are common  eigenfunctions of $d$ difference operators acting on the variables $m_1,m_2,\dots,m_d$, with coefficients independent of $\mt_1,\mt_2,\dots,\mt_d$. 

Formulas \eqref{6.1} show that the polynomials $\cP(m,\mt)$ form a remarkable basis. A random basis of polynomials orthogonal with respect to the the multinomial distribution will not satisfy such simple recurrence relations, see \cite{DX} for the general theory. If we add equations \eqref{6.1a} for $i=1,2,\dots,d$ and use the matrix equation \eqref{1.2} we obtain
\begin{equation}\label{6.3}
\begin{split}
-|m|\cP(m,\mt)=&p_0\sum_{l=1}^{d}\mt_l\cP(m,\mt-v_l)+\sum_{l=1}^{d}p_l(N-|\mt|)\cP(m,\mt+v_l)\\
&+\sum_{1\leq k\neq l\leq d}p_k\mt_l \cP(m,\mt+v_k-v_l)\\
&+\left(\sum_{j=1}^{d}\mt_jp_j+p_0(N-|\mt|)-N\right)\cP(m,\mt).
\end{split}
\end{equation}
Note that the coefficients on the right-hand side in the above partial difference equation depend only on $p$, $N$ and $\mt$, while the eigenvalue on the left-hand side depends only on the total degree $|m|$ (if we think of $\cP(m,\mt)$ as a polynomial of $\mt$). This equation is {\em universal\/} since every basis of polynomials orthogonal with respect to the multinomial distribution 
$$\frac{p_{0}^{N-|\mt|}p_1^{\mt_1}p_2^{\mt_2}\cdots p_{d}^{\mt_{d}}}{(N-|\mt|)!\mt_1!\mt_2!\cdots \mt_d!}$$ will satisfy \eqref{6.3}, see \cite{IX}.
\end{Remark}

\section{Explicit examples}

\subsection{}\label{ss7.1}
Let $d=2$ and define
\begin{align*}
&    u_{1,1}=1-\frac{(\fp_1+\fp_2)(\fp_1 +\fp_3)}{\fp_1(\fp_1+\fp_2+\fp_3+\fp_4)}, 
&& u_{1,2}=1-\frac{(\fp_1+\fp_2)(\fp_2 +\fp_4)}{\fp_2(\fp_1+\fp_2+\fp_3+\fp_4)},\\
&  u_{2,1}=1-\frac{(\fp_1+\fp_3)(\fp_3 +\fp_4)}{\fp_3(\fp_1+\fp_2+\fp_3+\fp_4)},
&& u_{2,2}=1-\frac{(\fp_2+\fp_4)(\fp_3 +\fp_4)}{\fp_4(\fp_1+\fp_2+\fp_3+\fp_4)},
\end{align*}

\begin{align*}
&    p_1=\frac{\fp_1 \fp_2 (\fp_1+\fp_2+\fp_3+\fp_4)}{(\fp_1+\fp_2)(\fp_1+\fp_3)(\fp_2+\fp_4)}, 
&& p_2= \frac{\fp_3 \fp_4 (\fp_1+\fp_2+\fp_3+\fp_4)}{(\fp_1+\fp_3)(\fp_2+\fp_4)(\fp_3+\fp_4)},\\
&\pt_1=\frac{\fp_1 \fp_3 (\fp_1+\fp_2+\fp_3+\fp_4)}{(\fp_1+\fp_2)(\fp_1+\fp_3)(\fp_3+\fp_4)},
&&\pt_2= \frac{\fp_2 \fp_4 (\fp_1+\fp_2+\fp_3+\fp_4)}{(\fp_1+\fp_2)(\fp_2+\fp_4)(\fp_3+\fp_4)},
\end{align*}
and 
$$p_0=\pt_0=1-p_1-p_2,$$
where the numbers $\{\fp_j\}_{1\leq j\leq 4}$ are essentially arbitrary, although certain combinations are forbidden in order to avoid dividing by $0$. A straightforward computation shows that \eqref{1.2} holds and we obtain the polynomials defined by Hoare and Rahman \cite{HR}. The constructions in the present paper correspond to the ones in  \cite{IT}. The bispectral involution $\fb$ in this case amounts to exchanging the parameters $\fp_2$ and $\fp_3$. We note also that independently of  \cite{IT}, Gr\"unbaum and Rahman \cite{GR} used different methods to derive an interesting recurrence relation for these polynomials which was suggested in \cite{Gr}.

\subsection{}\label{ss7.2}
Within the context of orthogonal polynomials one is often interested in explicit formulas for the polynomials 
in terms of the parameters which appear in the measure. We construct one such example which gives the multivariate Krawtchouk polynomials discovered by Milch~\cite{M}.

Let us fix  nonzero complex numbers $p_0,p_1,\dots,p_{d}$ such that $p_0+p_1+\cdots+p_{d}=1$. We define $\pt_0,\pt_1,\dots,\pt_{d}$ in terms of  $p_0,p_1,\dots,p_{d}$  as follows:
\begin{subequations}\label{7.1}
\begin{align}
&\pt_0=p_{0}\label{7.1a}\\
&\pt_{k}=\frac{p_kp_0}{(1-\sum_{j=1}^{k}p_j)(1-\sum_{j=1}^{k-1}p_j)}\text{ for }k=1,2,\dots,d.\label{7.1b}
\end{align}
\end{subequations}
Next we define a $(d+1)\times(d+1)$ matrix $U=(u_{i,j})$ with entries
\begin{subequations}\label{7.2}
\begin{align}
&u_{i,j}=\delta_{0,i},  &\text{ when }i<j,\label{7.2a}\\
&u_{i,j}=1,                 &\text{ when }i>j,\label{7.2b}\\
&u_{0,0}=1,\label{7.2c}\\
&u_{i,i}=-\frac{1-\sum_{k=1}^{i}p_k}{p_{i}}, &\text{ for }i=1,2,\dots,d.\label{7.2d}
\end{align}
\end{subequations}
Thus, $U$ is a matrix of the form \eqref{1.1} where the remaining entries are 0's and 1's above and below the diagonal respectively, and the diagonal entries are given in \eqref{7.2d}. One can check that with the above notations equation \eqref{1.2} holds and therefore we obtain explicit formulas for all constructions in the paper. Up to a permutation of the variables and the parameters, this choice leads to the multivariate Krawtchouk polynomials considered in \cite{M} and the corresponding bispectral difference operators constructed in \cite[Section~5.4]{GI}.

\subsection{}\label{ss7.3}
Fix now $q\notin\{0, 1\}$ and define $p_0, p_1,\dots,p_{d}$ as follows:
\begin{subequations}\label{7.3}
\begin{align}
&p_0=q^{-d}\label{7.3a}\\
&p_{k}=q^{-d+k-1}(q-1)\text{ for }k=1,2,\dots,d.\label{7.3b}
\end{align}
\end{subequations}
We denote
\begin{equation}\label{7.4}
P=\Pt=\diag(p_0,p_1,\dots,p_{d}),
\end{equation}
and we consider also a $(d+1)\times(d+1)$ matrix $U=(u_{i,j})$ with entries
\begin{subequations}\label{7.5}
\begin{align}
&u_{i,j}=1,  &\text{ when }i+j\leq d,\label{7.5a}\\
&u_{i,j}=\frac{1}{1-q},                 &\text{ when }i+j=d+1,\label{7.5b}\\
&u_{i,j}=0,&\text{ when }i+j>d+1.\label{7.5c}
\end{align}
\end{subequations}
With the above notations, we see that \eqref{1.2} holds leading to the family of multivariate Krawtchouk polynomials introduced in \cite{DS}. 

\section*{Acknowledgments}
I would like to thank Paul Terwilliger for suggestions to improve an earlier version of this paper.

\end{document}